\documentclass[11pt]{article}
\usepackage{amssymb,amsmath,graphicx,amsthm,mathrsfs}
\newtheorem{Theorem}{Theorem}[section]
\newtheorem{Lemma}[Theorem]{Lemma}
\newtheorem{Proposition}[Theorem]{Proposition}

\newtheorem{Remark}[Theorem]{Remark}
\numberwithin{equation}{section}

 \def\p{\partial} 
\def \Vh0{\stackrel{\circ}{V}_h}

\newcommand{\q}{\quad}    
\def\ll{\label}    

\newcommand{\lc}
{\mathrel{\raise2pt\hbox{${\mathop<\limits_{\raise1pt\hbox
{\mbox{$\sim$}}}}$}}}

\newcommand{\gc}
{\mathrel{\raise2pt\hbox{${\mathop>\limits_{\raise1pt\hbox{\mbox{$\sim$}}}}$}}}

\newcommand{\ec}
{\mathrel{\raise2pt\hbox{${\mathop=\limits_{\raise1pt\hbox{\mbox{$\sim$}}}}$}}}

\def\bb{\begin{equation}} \def\ee{\end{equation}}

\def\beqn{\begin{eqnarray}}  \def\eqn{\end{eqnarray}}

\def\beqnx{\begin{eqnarray*}} \def\eqnx{\end{eqnarray*}}

\def\bn{\begin{enumerate}} \def\en{\end{enumerate}}

\def\bd{\begin{description}} \def\ed{\end{description}}

\def \om {\Omega}
\def \ra {\rightarrow}
\def \o  {\overline}
\def \p  {\partial}

\def \d  {\displaystyle}

\def \e  {\varepsilon}

\title{Bang-bang property of  time optimal controls \\ for some semilinear heat equation}

\author{Lijuan Wang\thanks{School of
Mathematics and Statistics, Wuhan University, Wuhan 430072, China;
e-mail: ljwang.math@whu.edu.cn. This work was supported by the
National Natural Science Foundation of China under grant 11371285.} \q and \q Qishu Yan\thanks{School of Mathematics
and Statistics, Wuhan University, Wuhan 430072, China. e-mail:
yanqishu@whu.edu.cn.}}

\begin{document}

\date{}

\maketitle

\begin{abstract} In this paper, we derive a bang-bang property of  a kind of time optimal control problem
for some semilinear heat equation on bounded $C^2$ domains (of the Euclidean space),
with homogeneous Dirichlet boundary condition and controls distributed on an open and non-empty subset
of the domain where the equation evolves.  \\


\end{abstract}

\section{Introduction}

In this paper, we assume that $\om$ is a bounded domain in $\mathbb{R}^d, d\geq 1$, with a
$C^2$ boundary $\partial\Omega$; $\omega$ is an open and non-empty subset of $\om$; $\chi_{\omega}$ is the
characteristic function of the set $\omega$. Define, for an arbitrarily fixed $M>0$,
$q\in (d,\infty)$ if $d\geq 2$ and $q\in [2,\infty)$ if $d=1$,
a control constraint  set:
\begin{eqnarray*}
{\cal U}\triangleq\{u: [0,+\infty)\ra L^q(\om)\;\mbox{is
measurable}\; :\; \;\|u(t)\|_{L^q(\om)}\leq M\;\mbox{for almost
all}\;t>0\}.
\end{eqnarray*}
Let
$f:\mathbb{R}\ra \mathbb{R}$ be a function holding the following properties:\\

\noindent $(H_1)$ $f:\mathbb{R}\ra \mathbb{R}$ is locally Lipschitz;\\
\noindent $(H_2)$ $f(y) y\geq 0$ for all $y\in
\mathbb{R}$.\\

\noindent We arbitrarily fix a  $y_0\in C_0(\om)\setminus\{0\}$.
Here $C_0(\om)\triangleq\{y\in C(\o \om): y=0\;\mbox{on}\;\partial\om\}$.

The controlled semilinear heat equation  under consideration is as follows:
\begin{equation}
\left\{
\begin{array}{ll}
y_t-\Delta y+f(y)=\chi_{\omega} u&\mbox{in}\;\;\om\times
(0,+\infty),\\
y=0&\mbox{on}\;\;\p \om\times (0,+\infty),\\
y(0)=y_0&\mbox{in}\;\;\om,
\end{array}\right.\ll{eqn:1}\end{equation}
where  $u\in \mathcal{U}$. For each $u\in \mathcal{U}$ and $T>0$, Equation (\ref{eqn:1}) has a unique solution in $C([0,T];C_0(\om))$,
denoted it by $y(\cdot;y_0,u)$ (see Proposition~\ref{Energy:2}).
Throughout the paper, we will omit the
variables $x$ and $t$ for functions of $(x,t)$ and  the
variable $x$ for functions of $x$, if there is no risk of causing any
confusion.

Define the following admissible set of controls:
\begin{eqnarray*}
{\cal U}_{ad}\triangleq\{u\in \mathcal{U}: y(\cdot;y_0,u)\in C([0,T];C_0(\om))\;\mbox{and}\;
y(T;y_0,u)=0\;\mbox{for some}\;T>0\}.
\end{eqnarray*}
Each control in $\mathcal{U}_{ad}$ is called an admissible control.
Because of the properties $(H_1)$ and $(H_2)$, it is proved that the
set ${\cal U}_{ad}$ is not empty (see Proposition~\ref{Exist}). For
each $u\in \mathcal{U}_{ad}$, we set
\begin{equation*}
T(u)\triangleq \min\{T:\; y(T;y_0,u)=0\}.
\end{equation*}
The above minimum can be reached because the solution $y(\cdot;y_0,u)$ can be treated as
a continuous function from $[0,+\infty)$ to $C_0(\om)$.

Now, the time optimal control problem under consideration is as follows:
\begin{eqnarray*}
(P)\;\;\;\;\;T^*\triangleq\inf_{u\in \mathcal{U}_{ad}} T(u).
\end{eqnarray*}
In this problem, the number $T^*$ is called the optimal time; a
control  $u^*\in {\cal U}_{ad}$, with $y(T^*;y_0,u^*)=0$, is called a time optimal control (or optimal control
for simplicity). It is proved that $(P)$ has  optimal controls
(see Proposition~\ref{Optimal}).

The main result of this paper is stated as follows:
\begin{Theorem}\label{Bang} The problem $(P)$ holds the bang-bang property: any optimal control $u^*$ satisfies that
$\|u^*(t)\|_{L^q(\om)}=M$ for a.e. $t\in
(0,T^*)$.
\end{Theorem}

The bang-bang property is one of the most important and interesting
properties of time optimal control problems.
 To our best knowledge, this property was first built up, via a smart construction manner,
 for time optimal control problems of linear abstract  equations in Banach spaces by H.O.Fattorini (see \cite{Fattorini:1}).
  But in the
context of the distributed control of the heat equation, the results
in \cite{Fattorini:1} only apply for the very special case where the
control is distributed everywhere in the domain, i.e.
$\omega=\Omega$.
 Since then, such property has been studied for time optimal control problems of linear and semilinear
parabolic differential equations,  where  controls are  distributed everywhere
in the domain $\Omega$, in many papers (see e.g.
\cite{Barbu:1}, \cite{Barbu:2}, \cite{Fattorini:2}, \cite{Lions},
\cite{Wang-Wang} and the references therein).
 It is worth mentioning the following studies on the bang-bang property for time optimal control problems of
 parabolic equations with control restricted on a subset $\omega$ of $\Omega$.
  In \cite{Schmidt}, the ¡°bang-bang¡± property was obtained for a time optimal control problem of
 the heat equation with pointwise boundary
control constraints. Some assumptions on the control bound was imposed there.
 Partially motivated by the study \cite{Schmidt}, the authors in  \cite{Mizel} first realized  that
 the bang-bang property can be implied by the null-controllability of the system with controls
 restricted over any set of positive measure in a time
 interval. They further  proved such null-controllability for  one-dimensional
heat equation with boundary controls.
The aforementioned null-controllability was established  in \cite{Wang} for internally
controlled heat equations over domains $\Omega\subset\mathbb{R}^d$, $d\geq 1$,
As a consequence, the bang-bang property
holds for time optimal control problems of the corresponding equations.
Motivated by these results, the authors in \cite{Phung-wang-zhang}, \cite{Kunisch-wang:2}
and \cite{Kunisch-wang:3} realized that the bang-bang property for time optimal controls of
semilinear heat equations can be obtained by combining a strategy based on null controllability
of the system, where the control functions act on a measurable set, and a fixed point argument.
The main differences between this work and  \cite{Phung-wang-zhang},  \cite{Kunisch-wang:2}
, \cite{Kunisch-wang:3} are as follows: In \cite{Phung-wang-zhang}, the semilinear term is
globally Lipschitz and satisfies the sign condition $(H_2)$. Therefore, the bang-bang property is global, i.e., it holds
for any initial datum and any control bound. In \cite{Kunisch-wang:2} and \cite{Kunisch-wang:3}, the controlled equations are
Burgers equation and a general semilinear heat equation without sign condition $(H_2)$. There the bang-bang property is local
with respect to the initial datum and the control bound. In our paper, the semilinear
term is locally Lipschitz with sign condition, and the global bang-bang property is derived.
Hence, this paper is a nontrivial extension of the above-mentioned works.
It should be pointed out that if $y_0\in L^\infty(\om)$, by the similar arguments as
those in this paper, the global bang-bang property still holds.

It should be pointed out that in \cite{Wang}, \cite{Phung-wang-zhang},  \cite{Kunisch-wang:2}, \cite{Kunisch-wang:3} and this paper,
the controlled systems are time invariant and targets are a single point in state space.
The proofs rely on the translation invariant of the systems. This method does
not work for the time varying systems. In \cite{Wang-Xu-Zhang}, when the controlled systems read as follows:
\begin{equation*}
\left\{
\begin{array}{ll}
y_t-\Delta y+(a_1(x)+a_2(t)) y=\chi_{\omega} u&\mbox{in}\;\;\om\times
(0,+\infty),\\
y=0&\mbox{on}\;\;\p \om\times (0,+\infty),\\
y(0)=y_0&\mbox{in}\;\;\om,
\end{array}\right.\end{equation*}
and the target was a single point, the bang-bang property of time optimal controls was obtained.
The strategy used there was totally different from the existing works. Instead of studying directly the time optimal
control problem, the authors started from studying properties of optimal norms for norm optimal control problems.
However, when the potential $a_1(x)+a_2(t)$ is replaced by $a(x,t)$, whether the bang-bang property holds
or not is still an open problem.

When the target set is a ball, there is another possible way to
get bang-bang properties for time optimal control problems of differential equations. More precisely,
one can expect to derive bang-bang properties from the Pontryagin maximum principle
and unique continuation property for the corresponding equations.
Here, we would like to mention some related papers \cite{Wang-Wang1}, \cite{Kunisch} and \cite{Wang-Kunisch}.

To prove  Theorem~\ref{Bang},
we need an observability estimate from a measurable set
for heat equations with lower terms, and then use
 the Kakutani fixed point theorem. To this end, we introduce
 the following heat equation:
\begin{equation}\label{eqn:3}
\left\{
\begin{array}{lll}
\varphi_t-\Delta \varphi+a\varphi=0&\mbox{in}&\Omega\times (0,T),\\
\varphi=0&\mbox{on}&\partial\Omega\times (0,T),\\
\varphi(0)\in L^2(\Omega).&&
\end{array}\right.
\end{equation}
Here, $a\in L^\infty(\Omega\times (0,T))$.
Then
\begin{Lemma}\label{eqn:4}\cite{Phung-wang-zhang}
Let $E\subset (0,T)$ be a subset of positive measure. Then any solution to Equation (\ref{eqn:3})
satisfies the following observability estimate:
\begin{equation}\label{eqn:5}
\|\varphi(T)\|_{L^2(\Omega)}\leq e^{C(\Omega,\omega,E)} e^{C(\Omega,\omega)(1+T\|a\|_{L^\infty(\Omega\times (0,T))}
+\|a\|_{L^\infty(\Omega\times (0,T))}^{2/3})}
\int_{\omega\times E} |\varphi(x,t)|\,dx\,dt.
\end{equation}
\end{Lemma}
Here and throughout this paper, we shall use $C(\dots)$ to denote several positive constants
depending on what are enclosed in the bracket.
\begin{Remark}
The estimate (\ref{eqn:5}) is an observability inequality from a measurable set.
It was first established  for the case that $\Omega$ is convex in \cite{Phung}.
In \cite{Phung-wang-zhang},
 the convexity assumption on $\Omega$ was successfully dropped. Later on, the
 regularity assumption on the potential $a$ was relaxed to $a\in L^\infty(0,T;L^q(\om))$ in \cite{Kunisch-wang:2}.
\end{Remark}

The rest of the paper is organized as follows:  In Section 2, we provide some
existence results and prior estimates of solutions for parabolic equations. These results are required
to obtain the bang-bang property of the problem $(P)$. In Section 3, we
give the  proof of Theorem~\ref{Bang}.

\section{Preliminary results}
In this section, we shall present three results that will be essential for
the proof of the bang-bang property.
The first and the second results are concerned with the existence and  energy estimates of solutions
for  linear heat equations.

\begin{Lemma}\label{Energy:0} Let $g\in L^q(0,T;L^q(\om))$
and $z_0\in C_0(\om)$. Then the equation
\begin{equation}\ll{Energy:0-1}
\left\{
\begin{array}{lll}
z_t-\Delta z=g&\mbox{in}&\Omega\times (0,T),\\
z=0&\mbox{on}&\p\om\times (0,T),\\
z(0)=z_0&\mbox{in}&\om
\end{array}
\right.
\end{equation}
has a unique solution, denoted by $z$, in
$$
L^2(0,T;H_0^1(\om))\cap W^{1,2}(0,T;H^{-1}(\om))\cap C([0,T];C_0(\om)).
$$
Moreover,
\begin{equation*}
\|z\|_{C([0,T];C(\o \om))}\leq C(T) (\|g\|_{L^q(0,T;L^q(\om))}+\|z_0\|_{C(\o \om)}).
\end{equation*}
\end{Lemma}

\begin{proof}~It is obvious that (\ref{Energy:0-1}) has a unique solution, denoted by $z$,
in $L^2(0,T;H_0^1(\om))\cap W^{1,2}(0,T;H^{-1}(\om))$. Moreover,
\begin{equation}\ll{Energy:0-3}
z=z_1+z_2,
\end{equation}
where $z_1$ and $z_2$ are solutions to equations
\begin{equation*}
\left\{
\begin{array}{lll}
(z_1)_t-\Delta z_1=g&\mbox{in}&\Omega\times (0,T),\\
z_1=0&\mbox{on}&\p\om\times (0,T),\\
z_1(0)=0&\mbox{in}&\om
\end{array}
\right.
\end{equation*}
and
\begin{equation*}
\left\{
\begin{array}{lll}
(z_2)_t-\Delta z_2=0&\mbox{in}&\Omega\times (0,T),\\
z_2=0&\mbox{on}&\p\om\times (0,T),\\
z_2(0)=z_0&\mbox{in}&\om,
\end{array}
\right.
\end{equation*}
respectively. On one hand, it follows from Sobolev embedding theorem and $L^p-$theory for parabolic equation
(See e.g. Theorem 9.1 of Chapter 4 in \cite{Ladyzenskaja}) that
\begin{equation}\label{Energy:0-4}
\|z_1\|_{C([0,T];C(\o \om))}\leq C(T)\|z_1\|_{C([0,T];W_0^{1,q}(\om))}
\leq C(T) \|g\|_{L^q(0,T;L^q(\om))}.
\end{equation}
On the other hand, since $\Delta$ is the infinitesimal generator of a $C_0$ semigroup
on $C_0(\om)$ (See e.g. Theorem 3.7 of Chapter 7 in \cite{Pazy}), we have that
\begin{equation}\label{Energy:0-4-1}
\|z_2\|_{C([0,T];C(\o \om))}\leq C(T)\|z_0\|_{C(\o \om)},
\end{equation}
which, combined with (\ref{Energy:0-4}) and  (\ref{Energy:0-3}), completes the proof.
\end{proof}

Based on Lemma~\ref{Energy:0}, we have the following result.
\begin{Proposition}\label{Energy:1} Let $a\in L^\infty(\om\times (0,T)), g\in L^q(0,T;L^q(\om))$
and $z_0\in C_0(\om)$. Then the equation
\begin{equation}\ll{Energy:1-1-0}
\left\{
\begin{array}{lll}
z_t-\Delta z+a z=g&\mbox{in}&\Omega\times (0,T),\\
z=0&\mbox{on}&\p\om\times (0,T),\\
z(0)=z_0&\mbox{in}&\om
\end{array}
\right.
\end{equation}
has a unique solution, denoted by $z$, in
$$
L^2(0,T;H_0^1(\om))\cap W^{1,2}(0,T;H^{-1}(\om))\cap C([0,T];C_0(\om)).
$$
Moreover,
\begin{equation*}
\|z\|_{C([0,T];C(\o \om))}\leq C(T,\|a\|_{L^\infty(\om\times (0,T))}) (\|g\|_{L^q(0,T;L^q(\om))}+\|z_0\|_{C(\o \om)}).
\end{equation*}
\end{Proposition}
\begin{proof}~It is obvious that (\ref{Energy:1-1-0}) has a unique solution, denoted by $z$,
in $L^2(0,T;H_0^1(\om))\cap W^{1,2}(0,T;H^{-1}(\om))$. Moreover,
\begin{equation}\ll{Energy:1-1-2}
z=z_1+z_2,
\end{equation}
where $z_1$ and $z_2$ are solutions to equations
\begin{equation*}
\left\{
\begin{array}{lll}
(z_1)_t-\Delta z_1+a z_1=g&\mbox{in}&\Omega\times (0,T),\\
z_1=0&\mbox{on}&\p\om\times (0,T),\\
z_1(0)=0&\mbox{in}&\om
\end{array}
\right.
\end{equation*}
and
\begin{equation}\ll{Energy:1-1-3}
\left\{
\begin{array}{lll}
(z_2)_t-\Delta z_2+a z_2=0&\mbox{in}&\Omega\times (0,T),\\
z_2=0&\mbox{on}&\p\om\times (0,T),\\
z_2(0)=z_0&\mbox{in}&\om,
\end{array}
\right.
\end{equation}
respectively.

On one hand, it follows from Sobolev embedding theorem and $L^p-$theory for parabolic equation
(See e.g. Theorem 9.1 of Chapter 4 in \cite{Ladyzenskaja}) that
\begin{equation}\label{Energy:1-2}
\|z_1\|_{C([0,T];C(\o \om))}\leq C(T)\|z_1\|_{C([0,T];W_0^{1,q}(\om))}
\leq C(T,\|a\|_{L^\infty(\om\times (0,T))}) \|g\|_{L^q(0,T;L^q(\om))}.
\end{equation}
On the other hand, we set
\begin{equation}\label{Energy:1-2-1}
w(t)\triangleq e^{-\|a\|_{L^\infty(\Omega\times (0,T))} t} z_2(t),\;\;\;\forall\;t\in [0,T].
\end{equation}
One can easily check that
\begin{equation}\label{Energy:1-3}
\left\{
\begin{array}{lll}
w_t-\Delta w+(a+\|a\|_{L^\infty(\Omega\times (0,T))}) w=0&\mbox{in}&\Omega\times (0,T),\\
w=0&\mbox{on}&\p\om\times (0,T),\\
w(0)=z_0&\mbox{in}&\om.
\end{array}
\right.
\end{equation}
Let $\varphi\in C([0,T];C_0(\om))$ be the unique solution to
\begin{equation}\label{Energy:1-4}
\left\{
\begin{array}{lll}
\varphi_t-\Delta \varphi=0&\mbox{in}&\Omega\times (0,T),\\
\varphi=0&\mbox{on}&\p\om\times (0,T),\\
\varphi(0)=|z_0|&\mbox{in}&\om.
\end{array}
\right.
\end{equation}
Then
\begin{equation}\label{Energy:1-4-1}
\varphi(x,t)\geq 0\;\;\mbox{for each}\;\;(x,t)\in \om\times (0,T),
\end{equation}
and by the same argument for (\ref{Energy:0-4-1}), we have that
\begin{equation}\label{Energy:1-5}
\|\varphi\|_{C([0,T];C(\o \om))}\leq C(T)\|z_0\|_{C(\o \om)}.
\end{equation}
By (\ref{Energy:1-3}) and (\ref{Energy:1-4}), we get that
\begin{equation}\label{Energy:1-6}
\left\{
\begin{array}{lll}
(w-\varphi)_t-\Delta (w-\varphi)+(a+\|a\|_{L^\infty(\Omega\times (0,T))})w=0&\mbox{in}&\Omega\times (0,T),\\
w-\varphi=0&\mbox{on}&\p\om\times (0,T),\\
(w-\varphi)(0)\leq 0&\mbox{in}&\om.
\end{array}
\right.
\end{equation}
Multiplying the first equation of (\ref{Energy:1-6}) by $(w-\varphi)^+$ and integrating it over $\om$, after
some simple calculations, we obtain that
\begin{equation*}
\frac{1}{2}\frac{d}{dt}\|(w-\varphi)^+\|_{L^2(\om)}^2
+\int_\om (a+\|a\|_{L^\infty(\Omega\times (0,T))})w(w-\varphi)^+\,dx\leq 0,\;\;
\mbox{a.e.}\;t\in (0,T).
\end{equation*}
Integrating the latter inequality over $(0,t)$ and by (\ref{Energy:1-4-1}), we have that
\begin{equation*}
\|(w-\varphi)^+(t)\|_{L^2(\om)}\leq 0\;\;\mbox{for}\;t\in (0,T),
\end{equation*}
which implies
\begin{equation*}
w(x,t)\leq \varphi(x,t)\;\;\mbox{for a.e.}\; (x,t)\in \om\times (0,T).
\end{equation*}
Similarly, $w(x,t)\geq -\varphi(x,t)$ for a.e. $(x,t)\in \om\times (0,T)$. Hence
\begin{equation}\label{Energy:1-7}
|w(x,t)|\leq \varphi(x,t)\;\;\mbox{for a.e.}\; (x,t)\in \om\times (0,T).
\end{equation}
This together with (\ref{Energy:1-2-1}) and (\ref{Energy:1-5}) indicates
\begin{equation*}
\|z_2\|_{L^\infty(\Omega\times (0,T))}\leq C(T,\|a\|_{L^\infty(\Omega\times (0,T))})\|z_0\|_{C(\o \om)},
\end{equation*}
which, combined with (\ref{Energy:1-1-3}) and Lemma~\ref{Energy:0}, implies
\begin{equation*}
\|z_2\|_{C([0,T];C(\o \om))}\leq C(T,\|a\|_{L^\infty(\Omega\times (0,T))})\|z_0\|_{C(\o \om)}.
\end{equation*}
By the latter inequality, (\ref{Energy:1-2}) and (\ref{Energy:1-1-2}), we complete
the proof.
\end{proof}

The next result is about the existence and energy estimates of solutions for semilinear heat equations.

\begin{Proposition}\label{Energy:2} Let  $g\in L^q(0,T;L^q(\om))$
and $z_0\in C_0(\om)$. The equation
\begin{equation}\ll{Energy:2-1}
\left\{
\begin{array}{lll}
z_t-\Delta z+f(z)=g&\mbox{in}&\Omega\times (0,T),\\
z=0&\mbox{on}&\p\om\times (0,T),\\
z(0)=z_0&\mbox{in}&\om
\end{array}
\right.
\end{equation}
has a unique solution, denoted by $z$, in $C([0,T];C_0(\om))$. Moreover,
\begin{equation}\ll{Energy:2-1-1}
\|z\|_{C([0,T];C(\o \om))}\leq C(T) (\|g\|_{L^q(0,T;L^q(\om))}+\|z_0\|_{C(\o \om)}).
\end{equation}
\end{Proposition}
\begin{proof}~The proof of this proposition is divided into three steps.\\

Step 1. To show the uniquness\\

\noindent Suppose that $z_1\in C([0,T];C_0(\om))$ and $z_2\in C([0,T];C_0(\om))$ are two solutions of (\ref{Energy:2-1}) on $[0,T]$. Then
\begin{equation}\ll{Energy:2-1-2}
z_1(t)-z_2(t)=\int_0^t S(t-s) (f(z_2(s))-f(z_1(s)))\,ds,\;\;\;\forall\;t\in [0,T].
\end{equation}
Here and throughout the proof of this proposition, $S(\cdot)$ is the $C_0$ semigroup generated by 
Dirichlet-Laplacian on $L^2(\om)$.  Denote
\begin{equation*}
K\triangleq \max\{\|z_1\|_{C([0,T];C(\o \om))}, \|z_2\|_{C([0,T];C(\o \om))}\}+1,
\end{equation*}
and let $L>0$ be the Lipschitz constant of $f$ on $[-K,K]$. Then it follows from (\ref{Energy:2-1-2}) that
\begin{eqnarray*}
\|(z_1-z_2)(t)\|_{C(\o \om)}&=&\left\|\d{\int_0^t} S(t-s) (f(z_1(s))-f(z_2(s)))\,ds\right\|_{C(\o \om)}\\
&\leq&\d{\int_0^t} \|f(z_1(s))-f(z_2(s))\|_{C(\o \om)}\,ds\\
&\leq&L\d{\int_0^t} \|z_1(s)-z_2(s)\|_{C(\o \om)}\,ds,\;\;\;\forall\;t\in [0,T],
\end{eqnarray*}
which, combined with Gronwall's inequality, indicates the uniqueness.\\

Step 2. To prove the existence\\

By Lemma~\ref{Energy:0}, we have that
\begin{equation*}
\int_0^t S(t-s)g(s)\,ds\in C([0,T];C_0(\om)).
\end{equation*}
Let  $\widetilde{K}\triangleq\|z_0\|_{C(\o \om)}+1+\left\|\d{\int_0^t} S(t-s)g(s)\,ds\right\|_{C([0,T];C(\o \om))}$
and let
  $\widetilde{f}:\mathbb{R}\rightarrow \mathbb{R}$ be defined by:
\begin{equation}\ll{Energy:2-4}
\widetilde{f}(r)\triangleq\left\{
\begin{array}{lll}
f(\widetilde{K})&\mbox{if}&r>\widetilde{K},\\
f(r)&\mbox{if}&|r|\leq \widetilde{K},\\
f(-\widetilde{K})&\mbox{if}&r<-\widetilde{K}.
\end{array}\right.
\end{equation}
Then the function $\widetilde{f}$ is globally Lipschitz. By Lemma~\ref{Energy:0}, the following equation
\begin{equation}\ll{Energy:2-5}
\left\{
\begin{array}{lll}
\widetilde{w}_t-\Delta \widetilde{w}+\widetilde{f}(\widetilde{w})=g&\mbox{in}&\Omega\times (0,T),\\
\widetilde{w}=0&\mbox{on}&\p\om\times (0,T),\\
\widetilde{w}(0)=z_0&\mbox{in}&\om
\end{array}
\right.
\end{equation}
has a unique solution $\widetilde{w}\in C([0,T];C_0(\om))$. Moreover,
\begin{equation*}
\widetilde{w}(t)=S(t)z_0+\int_0^t S(t-s)(g(s)-\widetilde{f}(\widetilde{w}(s)))\,ds,\;\;\;\forall\;t\in [0,T].
\end{equation*}
Then
\begin{eqnarray*}
&&\|\widetilde{w}(t)\|_{C(\o \om)}\\
&\leq&\|S(t)z_0\|_{C(\o \om)}+\left\|\d{\int_0^t} S(t-s) g(s)\,ds\right\|_{C([0,T];C(\o \om))}
+\d{\int_0^t} \left\|\widetilde{f}(\widetilde{w}(s))\right\|_{C(\o \om)}\,ds\\
&\leq&\|z_0\|_{C(\o \om)}+\left\|\d{\int_0^t} S(t-s) g(s)\,ds\right\|_{C([0,T];C(\o \om))}
+t\|f\|_{C\left(\left[-\widetilde{K},\widetilde{K}\right]\right)},\;\;\;\forall\;t\in [0,T].
\end{eqnarray*}
Set $T_1=\min\left\{T,\|f\|_{C\left(\left[-\widetilde{K},\widetilde{K}\right]\right)}^{-1}\right\}$. From the latter it follows that
\begin{equation*}
\|\widetilde{w}\|_{C([0,T_1];C(\o \om))}\leq \widetilde{K},
\end{equation*}
which, combined with (\ref{Energy:2-4}) and (\ref{Energy:2-5}), indicates
the equation
\begin{equation}\ll{Energy:2-3}
\left\{
\begin{array}{lll}
\widetilde{z}_t-\Delta \widetilde{z}+f(\widetilde{z})=g&\mbox{in}&\Omega\times (0,T),\\
\widetilde{z}=0&\mbox{on}&\p\om\times (0,T),\\
\widetilde{z}(0)=z_0&\mbox{in}&\om.
\end{array}
\right.
\end{equation}
has a solution in $C([0,T_1];C_0(\om))$. Then by iterating this construction and
the uniqueness as Step 1, we obtain the existence of a solution $\widetilde{z}$ on $[0,T]$ or on a maximal time interval
$[0,T_{\mbox{max}})$ with $T_{\mbox{max}}\leq T$.\\

Case 1. $\widetilde{z}\in C([0,T];C_0(\om))$. In this case,  the proof is finished.\\

Case 2. $\widetilde{z}\in C([0,T_{\mbox{max}});C_0(\om))$. We claim that this case cannot happen.\\
By contradiction, on one hand, since $[0,T_{\mbox{max}})$ is the maximal time interval for the solution of
(\ref{Energy:2-3}) in $C([0,T_{\mbox{max}});C_0(\om))$, we have
\begin{equation}\label{Energy:2-7}
\limsup\limits_{t\uparrow T_{\mbox{max}}} \|\widetilde{z}(t)\|_{C(\o \om)}=+\infty.
\end{equation}
On the other hand, for any $\widetilde{T}\in (0,T_{\mbox{max}})$, we consider the following two equations:
\begin{equation}\ll{Energy:2-8}
\left\{
\begin{array}{lll}
\widetilde{z}_t-\Delta \widetilde{z}+f(\widetilde{z})=g&\mbox{in}&\Omega\times (0,\widetilde{T}),\\
\widetilde{z}=0&\mbox{on}&\p\om\times (0,\widetilde{T}),\\
\widetilde{z}(0)=z_0&\mbox{in}&\om
\end{array}
\right.
\end{equation}
and
\begin{equation}\ll{Energy:2-9}
\left\{
\begin{array}{lll}
\psi_t-\Delta \psi=|g|&\mbox{in}&\Omega\times (0,T),\\
\psi=0&\mbox{on}&\p\om\times (0,T),\\
\psi(0)=|z_0|&\mbox{in}&\om.
\end{array}
\right.
\end{equation}
It follows from (\ref{Energy:2-8}) and (\ref{Energy:2-9}) that
\begin{equation}\ll{Energy:2-10}
\left\{
\begin{array}{lll}
(\widetilde{z}-\psi)_t-\Delta (\widetilde{z}-\psi)+f(\widetilde{z})\leq 0&\mbox{in}&\Omega\times (0,\widetilde{T}),\\
\widetilde{z}-\psi=0&\mbox{on}&\p\om\times (0,\widetilde{T}),\\
(\widetilde{z}-\psi)(0)\leq 0&\mbox{in}&\om.
\end{array}
\right.
\end{equation}
Multiplying the first equation of (\ref{Energy:2-10}) by $(\widetilde{z}-\psi)^+$
and integrating it over $\om$, by $(H_2)$ and the similar arguments for (\ref{Energy:1-7}), we have that
\begin{equation}\label{Energy:2-10-1}
|\widetilde{z}(x,t)|\leq \psi(x,t)\;\;\mbox{a.e. in}\;(x,t)\in \om\times (0,\widetilde{T}),
\end{equation}
which, combined with (\ref{Energy:2-9}) and Lemma~\ref{Energy:0}, indicates
\begin{equation}\label{Energy:2-11}
\begin{array}{lll}
\|\widetilde{z}\|_{C\left(\left[0,\widetilde{T}\right];C(\o \om)\right)}
&\leq&\|\psi\|_{C\left(\left[0,\widetilde{T}\right];C(\o \om)\right)}\leq \|\psi\|_{C([0,T];C(\o \om))}\\
&\leq&C(T)(\|g\|_{L^q(0,T;L^q(\om))}+\|z_0\|_{C(\o \om)}).
\end{array}
\end{equation}
This contradicts (\ref{Energy:2-7}).\\

Step 3. To show (\ref{Energy:2-1-1})\\

\noindent Consider the following equation
\begin{equation*}
\left\{
\begin{array}{lll}
\phi_t-\Delta \phi=|g|&\mbox{in}&\Omega\times (0,T),\\
\phi=0&\mbox{on}&\p\om\times (0,T),\\
\phi(0)=|z_0|&\mbox{in}&\om.
\end{array}
\right.
\end{equation*}
By the similar arguments led to (\ref{Energy:2-11}), we have
\begin{equation*}
\|z\|_{C([0,T];C(\o \om))}\leq \|\phi\|_{C([0,T];C(\o \om))}
\leq C(T)(\|g\|_{L^q(0,T;L^q(\om))}+\|z_0\|_{C(\o \om)}).
\end{equation*}

This completes the proof.
\end{proof}

\section{Proof of Theorem~\ref{Bang}}
 In this section, we shall first present the following local null controllability of a semilinear heat equation.
\begin{Proposition}\label{Control} Let $E\subset (0,T)$ with $|E|>0$. For any $\phi\in L^\infty(\om\times (0,T))$,
there are two positive constants $\rho_0\triangleq \rho_0(E,T,f,\|\phi\|_{L^\infty(\om\times (0,T))})$
and $\kappa\triangleq \kappa(E,T,f,\|\phi\|_{L^\infty(\om\times (0,T))})$, so that
for any $w_0\in C_0(\om)$ with $\|w_0\|_{C(\o \om)}\leq \rho_0$,
there exists
a function $v\in L^\infty(\om\times (0,T))$ so that
\begin{equation*}
\|v\|_{L^\infty(\om\times (0,T))}\leq \kappa \|w_0\|_{L^2(\om)},
\end{equation*}
and $w\in C([0,T];C_0(\om))$ satisfies
\begin{equation*}
\left\{
\begin{array}{lll}
w_t-\Delta w+f(\phi+w)-f(\phi)=\chi_\omega \chi_E v&\mbox{in}&\om\times (0,T),\\
w=0&\mbox{on}&\partial\om\times (0,T),\\
w(0)=w_0&\mbox{in}&\om,\\
w(T)=0&\mbox{in}&\om.
\end{array}\right.
\end{equation*}
\end{Proposition}

\begin{proof}~By a classical density argument,
we may assume that $f\in C^1$. We shall use the Kakutani's Fixed Point Theorem to prove it.
To this end, for any $(x,t)\in \om\times (0,T)$, we define
\begin{equation*}
a(x,t,r)\triangleq\left\{
\begin{array}{lll}
\frac{f(\phi(x,t)+r)-f(\phi(x,t))}{r}&\mbox{if}&r\not=0,\\
f^\prime(\phi(x,t))&\mbox{if}&r=0.
\end{array}\right.
\end{equation*}
Set
\begin{equation*}
\mathcal{K}\triangleq\{\xi\in L^2(0,T;L^2(\Omega)):\;\|\xi\|_{L^2(0,T;H_0^1(\Omega))\cap W^{1,2}(0,T;H^{-1}(\om))}
+\|\xi\|_{L^\infty(\om\times (0,T))}\leq 1\}.
\end{equation*}
For each $\xi\in \mathcal{K}$, consider the linear control system
\begin{equation}\label{Control:1}
\left\{
\begin{array}{lll}
w_t-\Delta w+a(x,t,\xi(x,t))w=\chi_\omega \chi_E v&\mbox{in}&\Omega\times (0,T),\\
w=0&\mbox{on}&\partial\Omega\times (0,T),\\
w(0)=w_0&\mbox{in}&\Omega.
\end{array}\right.
\end{equation}
By $(H_1)$, we have that
\begin{equation}\label{Control:2}
|a(x,t,\xi(x,t))|\leq C_1\;\;\mbox{a.e.}\; (x,t)\in \om\times (0,T),
\end{equation}
where
\begin{equation*}
C_1\triangleq C_1(f,\|\phi\|_{L^\infty(\om\times (0,T))})
\end{equation*}
denotes the Lipschitz constant of $f$ on $[-\|\phi\|_{L^\infty(\om\times (0,T))}-1,\|\phi\|_{L^\infty(\om\times (0,T))}+1]$.
By Lemma~\ref{eqn:4}, (\ref{Control:1}) and (\ref{Control:2}), there exists
$$
\kappa\triangleq \kappa(E,T,f,\|\phi\|_{L^\infty(\om\times (0,T))})>0\;\;\mbox{
and}\;\;v\in L^\infty(\om\times (0,T))
$$
with
\begin{equation}\label{Control:2-1}
\|v\|_{L^\infty(\om\times (0,T))}\leq \kappa\|w_0\|_{L^2(\om)},
\end{equation}
so  that
\begin{equation}\label{Control:3}
w(T)=0.
\end{equation}

Now, we define the multivalued  map $\Phi: \mathcal{K}\rightarrow L^2(0,T;L^2(\Omega))$ by
\begin{equation*}
\Phi(\xi)\triangleq\{w:\;\mbox{there exists a control}\;\;v
\;\;\mbox{so that (\ref{Control:1}), (\ref{Control:2-1}) and (\ref{Control:3}) hold}\},\;\;\mbox{for}\;\;\xi\in \mathcal{K}.
\end{equation*}
From the above arguments it follows  that $\Phi(\xi)\not=\emptyset$ for each $\xi\in \mathcal{K}$.\\

Next we  check in three steps the conditions of Kakutani's fixed point
theorem.\\

\noindent Step 1. To show that $\mathcal{K}$ is a convex and
compact set in $L^2(0,T;L^2(\om))$ and $\Phi(\xi)$ is
a convex set in $L^2(0,T;L^2(\om))$ for each $\xi\in
\mathcal{K}$\\

These can be directly checked.\\

\noindent{\it Step 2. To show that $\Phi(\mathcal{K})\subset
\mathcal{K}$}\\

To achieve this goal, we observe that for every $\xi\in
\mathcal{K}$, there is   $v\in
L^\infty(\om\times (0,T))$, with the estimate
\begin{equation}\label{Control:5}
\|v\|_{L^\infty(\om\times (0,T))}\leq
\kappa\|w_0\|_{L^2(\Omega)},
\end{equation}
so that the associated state $w=w(x,t)$ satisfies
\begin{equation}\label{Control:6}
\left\{
\begin{array}{lll}
w_t-\Delta w+a(x,t,\xi(x,t))w=\chi_\omega \chi_E v&\mbox{in}&\Omega\times (0,T),\\
w=0&\mbox{on}&\partial\Omega\times (0,T),\\
w(0)=w_0&\mbox{in}&\Omega,\\
w(T)=0&\mbox{in}&\Omega.
\end{array}\right.
\end{equation}
The standard energy method, the fact that $|a|\leq C_1$ and Proposition~\ref{Energy:1}, as
well as  (\ref{Control:5}) and (\ref{Control:6}), lead to
\begin{equation*}
\|w\|_{L^2(0,T;H_0^1(\Omega))\cap
W^{1,2}(0,T;H^{-1}(\Omega))}+\|w\|_{C([0,T];C(\o \om))}\leq
C_2\|w_0\|_{C(\o \om)},
\end{equation*}
for some positive constant $C_2\triangleq C_2(E,T,f,\|\phi\|_{L^\infty(\om\times (0,T))})$.
Hence, if
$$
\|w_0\|_{C(\o \om)}\leq\rho_0\triangleq C_2^{-1},
$$
then $\Phi(\mathcal{K})\subset\mathcal{K}$.\\

\noindent{\it Step 3. To show that  $\Phi$ is upper semicontinuous
in $L^2(0,T;L^2(\Omega))$}\\

It suffices to show that if
$$
\xi_m\in \mathcal{K}\rightarrow \xi\;\;\mbox{ strongly
in}\;\;L^2(0,T;L^2(\Omega))
$$
and
$$
w_m\in
\Phi(\xi_m)\rightarrow w\;\;\mbox{ strongly
in}\;\;L^2(0,T;L^2(\Omega)),
$$
then $w\in
 \Phi(\xi)$.
 To this end,  we first observe that $\xi\in \mathcal{K}$. Next we
 claim that there exists a subsequence of
$\{m\}_{m\geq 1}$, still denoted in the same manner, so that
\begin{equation}\label{Control:7}
a(x,t,\xi_m)w_m\rightarrow a(x,t,\xi)w\;\;\mbox{strongly in}\;\;L^2(0,T;L^2(\Omega)).
\end{equation}
Indeed, since $\xi_m\rightarrow \xi$ strongly in $L^2(0,T;L^2(\Omega))$, we have  a subsequence of $\{m\}_{m\geq 1}$,
still denoted by itself, so that
\begin{equation*}
\xi_m(x,t)\rightarrow \xi(x,t)\;\;\mbox{for a.e.}\;\;(x,t)\in \Omega\times (0,T).
\end{equation*}
On one hand, for $(x,t)$ with $\xi(x,t)\not=0$, by the above, there exists a positive integer $m_0$ depending on $(x,t)$ so that
\begin{equation*}
\xi_m(x,t)\not=0\;\;\;\forall\;m\geq m_0,
\end{equation*}
which, combined with the definition of $a$, implies that
\begin{equation}\label{Control:8}
a(x,t,\xi_m(x,t))\rightarrow a(x,t,\xi(x,t))\;\;\mbox{as}\;\;m\rightarrow +\infty.
\end{equation}
On the other hand, for any $(x,t)$ satisfying $\xi(x,t)=0$, by the definition of $a$,
we have that $a(x,t,\xi(x,t))=f^\prime(\phi(x,t))$. Since
\begin{equation*}
a(x,t,\xi_m(x,t))=\left\{
\begin{array}{lll}
\frac{f(\phi(x,t)+\xi_m(x,t))-f(\phi(x,t))}{\xi_m(x,t)}&\mbox{if}&\xi_m(x,t)\not=0,\\
f^\prime(\phi(x,t))&\mbox{if}&\xi_m(x,t)=0,
\end{array}\right.
\end{equation*}
it holds that
\begin{equation*}
a(x,t,\xi_m(x,t))\rightarrow a(x,t,\xi(x,t))\;\;\mbox{as}\;\;m\rightarrow +\infty.
\end{equation*}
This, combined with (\ref{Control:8}), implies
\begin{equation*}
a(x,t,\xi_m(x,t))\rightarrow a(x,t,\xi(x,t))\;\;\mbox{for a.e.}\;\;(x,t)\in \Omega\times (0,T).
\end{equation*}
From the latter, the fact that $|a(x,t,\xi_m(x,t))|\leq C_1$ and  Lebesgue's
dominated convergence theorem, it follows that
\begin{eqnarray*}
&&\|a(x,t,\xi_m)w_m-a(x,t,\xi)w\|_{L^2(0,T;L^2(\Omega))}^2\\
&\leq&2\|a(x,t,\xi_m)(w_m-w)\|_{L^2(0,T;L^2(\Omega))}^2+2\|(a(x,t,\xi_m)-a(x,t,\xi))w\|_{L^2(0,T;L^2(\Omega))}^2\\
&\leq&2 C_1^2\|w_m-w\|_{L^2(0,T;L^2(\Omega))}^2+2\|(a(x,t,\xi_m)-a(x,t,\xi))w\|_{L^2(0,T;L^2(\Omega))}^2\\
&\rightarrow&0.
\end{eqnarray*}
This leads to  (\ref{Control:7}).

Finally,  since $w_m\in \Phi(\xi_m)\subset \mathcal{K}$, there are $v_m\in
L^\infty(\om\times (0,T))$, $m=1, 2, \dots$, satisfying
\begin{equation}\label{Control:9}
\|v_m\|_{L^\infty(\om\times (0,T))}\leq
\kappa\|w_0\|_{L^2(\Omega)}\;\;\mbox{for all}\;\;
m\geq 1,
\end{equation}
\begin{equation}\label{Control:10}
\left\{
\begin{array}{lll}
(w_m)_t-\Delta w_m+a(x,t,\xi_m(x,t))w_m
=\chi_\omega \chi_E v_m&\mbox{in}&\Omega\times (0,T),\\
w_m=0&\mbox{on}&\partial\Omega\times (0,T),\\
w_m(0)=w_0&\mbox{in}&\Omega,\\
w_m(T)=0&\mbox{in}&\Omega
\end{array}\right.
\end{equation}
and
\begin{equation}\label{Control:11}
\|w_m\|_{L^2(0,T;H_0^1(\Omega))\cap W^{1,2}(0,T;H^{-1}(\Omega))}+\|w_m\|_{L^\infty(\om\times (0,T))}\leq 1.
\end{equation}
Thus, there is a
control $v$ and a subsequence of $\{m\}_{m\geq 1}$, still denoted by
itself, so that
\begin{equation}\label{Control:12}
v_m\rightarrow v\;\;\mbox{weakly star in}\;\;L^\infty(\om\times (0,T)),
\end{equation}
\begin{equation}\label{Control:13}
\begin{array}{ll}
w_m\rightarrow w&\mbox{weakly in}\;\;L^2(0,T;H_0^1(\Omega))\cap W^{1,2}(0,T;H^{-1}(\Omega)),\\
&\mbox{weakly star in}\;\;L^\infty(\om\times (0,T)),
\end{array}
\end{equation}
and
\begin{equation}\label{Control:14}
w_m(T)\rightarrow w(T)\;\;\mbox{strongly in}\;\;L^2(\Omega).
\end{equation}
Finally, passing to the limit for $m\rightarrow +\infty$ in
(\ref{Control:9})-(\ref{Control:11}), making use of
(\ref{Control:7}) and (\ref{Control:12})-(\ref{Control:14}),
we obtain that $w\in \Phi(\xi)$.

Now, by conclusions in  Step 1-Step 3, we can apply  Kakutani's
fixed point theorem to get  a  function $w\in \mathcal{K}$ so that
$w\in \Phi(w)$, which, combined with Proposition~\ref{Energy:1}, indicates $w\in C([0,T];C_0(\om))$.
Moreover, since
\begin{equation*}
a(x,t,w(x,t))w(x,t)=f(\phi(x,t)+w(x,t))-f(\phi(x,t)),
\end{equation*}
the results follow at once. This completes
the proof of this proposition.
\end{proof}
\begin{Remark}~In this proposition, only the hypothesis $(H_1)$ is used.
\end{Remark}

The next proposition is concerned with the existence of  admissible controls for the problem $(P)$.

\begin{Proposition}\label{Exist} There exists an admissible control for the problem $(P)$.
\end{Proposition}
\begin{proof}~For any $T_0>0$ fixed, we consider the following equation
\begin{equation}\label{Exist:1}
\left\{
\begin{array}{lll}
y_t-\Delta y+f(y)=0&\mbox{in}&\Omega\times (0,T_0+1),\\
y=0&\mbox{on}&\partial\Omega\times (0,T_0+1),\\
y(0)=y_0&\mbox{in}&\Omega.
\end{array}\right.
\end{equation}
It follows from Proposition~\ref{Energy:2} that (\ref{Exist:1}) has a unique solution in
$C([0,T_0+1];C_0(\om))\cap L^2(0,T_0+1;H_0^1(\om))\cap W^{1,2}(0,T_0+1;H^{-1}(\om))$,
denoted it by $y(\cdot;y_0,0)$. Multiplying the first equation of (\ref{Exist:1}) by $y(t;y_0,0)$
and integrating it over $\om$, by $(H_2)$, we get that
$$
\frac{1}{2}\frac{d}{dt}\int_\om |y(t;y_0,0)|^2\,dx
+\int_\om |\nabla y(t;y_0,0)|^2\,dx\leq 0,
$$
which implies
\begin{equation*}
\frac{d}{dt}\|y(t;y_0,0)\|_{L^2(\om)}^2+2\lambda_1\|y(t;y_0,0)\|_{L^2(\om)}^2\leq 0\;\;\;\mbox{for a.e.} \;t\in (0,T_0+1).
\end{equation*}
Here $\lambda_1>0$ is the first eigenvalue of $-\Delta$ with zero boundary condition.
Multiplying the above inequality by $e^{2\lambda_1 t}$ and integrating it over $(0,t)$, after some calculations, we obtain that
\begin{equation}\label{Exist:2}
\|y(T_0;y_0,0)\|_{L^2(\om)}\leq e^{-\lambda_1 T_0}\|y_0\|_{L^2(\om)}
\end{equation}
and
\begin{equation}\label{Exist:2-1}
\|y(T_0+1;y_0,0)\|_{L^2(\om)}\leq e^{-\lambda_1 (T_0+1)}\|y_0\|_{L^2(\om)}.
\end{equation}

Set
\begin{equation}\ll{Exist:3-1}
w(t)\triangleq y(t+T_0;y_0,0),\;\;\;\forall\;t\in [0,1].
\end{equation}
Then $w\in C([0,1];C_0(\om))$ satisfies
\begin{equation*}
\left\{
\begin{array}{lll}
w_t-\Delta w+f(w)=0&\mbox{in}&\Omega\times (0,1),\\
w=0&\mbox{on}&\partial\Omega\times (0,1),\\
w(0)=y(T_0;y_0,0)&\mbox{in}&\Omega.
\end{array}\right.
\end{equation*}
By the similar arguments used to obtain (\ref{Energy:2-10-1}), we have
\begin{equation}\label{Exist:3-2}
\|w(1)\|_{C(\o \om)}\leq \|S(1) |y(T_0;y_0,0)|\|_{C(\o \om)}.
\end{equation}
Here and throughout this Proposition, $S(\cdot)$ is the $C_0$ semigroup generated by 
Dirichlet-Laplacian on $L^2(\om)$.
It follows from Proposition 4.4 of Chapter 1 in \cite{Barbu:1} that
\begin{equation*}
\|S(1)|y(T_0;y_0,0)|\|_{C(\o \om)}\leq C\|y(T_0;y_0,0)\|_{L^1(\om)}.
\end{equation*}
Here $C>0$ is a  constant independent of $T_0$.
From the latter inequality, (\ref{Exist:3-1}), (\ref{Exist:3-2}) and (\ref{Exist:2}), we obtain that
\begin{equation}\label{Exist:3-3}
\|y(T_0+1;y_0,0)\|_{C(\o \om)}\leq C|\om|^{\frac{1}{2}}e^{-\lambda_1 T_0} \|y_0\|_{L^2(\om)}.
\end{equation}

By Proposition~\ref{Control}, there exist two positive constants
$\rho_0\triangleq \rho_0(f)$ and $\kappa\triangleq \kappa(f)$ , so that
for any $\varphi_0\in C_0(\om)$ with $\|\varphi_0\|_{C(\o \om)}\leq \rho_0$,
there exists
a function $v\in L^\infty(0,1;L^q(\om))$ with
\begin{equation}\label{Exist:4}
\|v\|_{L^\infty(0,1;L^q(\om))}\leq \kappa \|\varphi_0\|_{L^2(\om)},
\end{equation}
so that $\varphi\in C([0,1];C_0(\om))$ satisfies
\begin{equation}\label{Exist:5}
\left\{
\begin{array}{lll}
\varphi_t-\Delta \varphi+f(\varphi)=\chi_\omega v&\mbox{in}&\om\times (0,1),\\
\varphi=0&\mbox{on}&\partial\om\times (0,1),\\
\varphi(0)=\varphi_0&\mbox{in}&\om,\\
\varphi(1)=0&\mbox{in}&\om.
\end{array}\right.
\end{equation}

Now, we take
\begin{equation*}
T_0\triangleq\frac{1}{\lambda_1}\left(\left|\ln\frac{C|\om|^{\frac{1}{2}}\|y_0\|_{L^2(\om)}}{\rho_0}\right|
+\left|\ln\frac{\kappa\|y_0\|_{L^2(\om)}}{M}\right|\right)+1.
\end{equation*}
Then it follows from (\ref{Exist:2-1}) and (\ref{Exist:3-3})-(\ref{Exist:5}) that there exists a control $u\in L^\infty(0,1;L^q(\om))$ with
\begin{equation*}
\|u\|_{L^\infty(0,1;L^q(\om))}\leq M,
\end{equation*}
so that $z\in C([0,1];C_0(\om))$ satisfies
\begin{equation*}
\left\{
\begin{array}{lll}
z_t-\Delta z+f(z)=\chi_\omega u&\mbox{in}&\om\times (0,1),\\
z=0&\mbox{on}&\partial\om\times (0,1),\\
z(0)=y(T_0+1;y_0,0)&\mbox{in}&\om,\\
z(1)=0&\mbox{in}&\om.
\end{array}\right.
\end{equation*}
Finally, it is easy to check the function defined by
\begin{equation*}
\o u(t)\triangleq\left\{
\begin{array}{ll}
0,&t\in (0,T_0+1),\\
u(t-T_0-1),&t\in (T_0+1,T_0+2),\\
0,&t\in (T_0+2,+\infty),
\end{array}\right.
\end{equation*}
is an admissible control for the problem $(P)$.
\end{proof}

Based on Proposition~\ref{Exist}, we shall show the existence of optimal solutions for the problem $(P)$.
\begin{Proposition}\label{Optimal} There exists at least one optimal solution for the problem $(P)$.
\end{Proposition}
\begin{proof}~Let $T^*\triangleq\mbox{inf}(P)$. It is obvious that $0\leq T^*<+\infty$.
Then there exist
sequences $\{T_n\}_{n\geq 1}$ and $\{u_n\}_{n\geq 1}\subset
\mathcal{U}$ so that
\begin{equation}\label{Optimal:1}
T^*=\lim_{n\rightarrow +\infty} T_n
\end{equation}
and
\begin{equation}\label{Optimal:2}
\left\{
\begin{array}{ll}
\partial_t y_n-\Delta y_n+f(y_n)=\chi_\omega u_n&\mbox{in}\;\;\Omega\times (0,T_n),\\
y_n=0&\mbox{on}\;\;\partial \Omega\times (0,T_n),\\
y_n(0)=y_0&\mbox{in}\;\;\Omega,\\
y_n(T_n)=0&\mbox{in}\;\;\Omega,
\end{array}\right.
\end{equation}
where $y_n(\cdot)\in C([0,T_n];C_0(\om))$.
By (\ref{Optimal:1}) and (\ref{Optimal:2}), we can assume that
$0<T_n<T^*+1$. Set
\begin{equation}\label{Optimal:3}
v_n(t)\triangleq\left\{\begin{array}{ll}
u_n(t) ,&t\in [0,T_n),\\
0,&t\in [T_n,+\infty)
\end{array}
\right.
\;\;\;\mbox{and}\;\;\;z_n(t)\triangleq\left\{
\begin{array}{ll}
y_n(t),&t\in [0,T_n),\\
0,&t\in [T_n,T^*+1].
\end{array}
\right.
\end{equation}
Since $\{u_n\}_{n\geq 1}\subset \mathcal{U}$, it follows from (\ref{Optimal:2}) and (\ref{Optimal:3})  that
\begin{equation}\label{Optimal:4}
\{v_n\}_{n\geq 1}\subset\mathcal{U}
\end{equation}
and that $z_n(\cdot)\in C([0,T^*+1];C_0(\om))$ satisfies
\begin{equation}\label{Optimal:5}
\left\{
\begin{array}{ll}
\partial_t z_n-\Delta z_n+f(z_n)=\chi_\omega v_n&\mbox{in}\;\;\Omega\times (0,T^*+1),\\
z_n=0&\mbox{on}\;\;\partial \Omega\times (0,T^*+1),\\
z_n(0)=y_0&\mbox{in}\;\;\Omega,\\
z_n(T_n)=0&\mbox{in}\;\;\Omega.
\end{array}\right.
\end{equation}
By (\ref{Optimal:5}), (\ref{Optimal:4}) and Proposition~\ref{Energy:2}, we obtain that
$$
\begin{array}{lll}
&&\|z_n\|_{L^2(0,T^*+1;H_0^1(\om))\cap W^{1,2}(0,T^*+1;H^{-1}(\om))}+\|z_n\|_{C([0,T^*+1];C(\o \om))}\\
&\leq&C(\|v_n\|_{L^q(0,T^*+1;L^q(\om))}+\|y_0\|_{C(\o \om)})\\
&\leq& C,\;\;\;\;\;\;\;\;\;\;\;\;\forall\;n\geq 1.
\end{array}
$$
Here $C>0$ is a constant independent of $n$.
This, combined with (\ref{Optimal:4}) and (\ref{Optimal:5}), implies  that there exists a subsequence of $\{n\}_{n\geq 1}$,  denoted
in the same manner, $z\in L^2(0,T^*+1;H_0^1(\om))\cap W^{1,2}(0,T^*+1;H^{-1}(\om))\cap L^\infty(\om\times (0,T^*+1))$ and
$v\in \mathcal{U}$, so that
\begin{equation}\label{Optimal:6}
\begin{array}{ll}
z_n\rightarrow z&\mbox{weakly in}\;\;L^2(0,T^*+1;H_0^1(\om))\cap W^{1,2}(0,T^*+1;H^{-1}(\om)),\\
&\mbox{weakly star in}\;L^\infty(\om\times (0,T^*+1))\;\mbox{and strongly in}\;C([0,T^*+1];L^2(\om)),\\
v_n\rightarrow v&\mbox{weakly star in}\;\;L^\infty(0,+\infty;L^q(\Omega)).
\end{array}
\end{equation}
Moreover, $v\in \mathcal{U}$ and $z$ satisfies
\begin{equation}\label{Optimal:7}
\left\{
\begin{array}{ll}
\partial_t z-\Delta z+f(z)=\chi_\omega v&\mbox{in}\;\;\Omega\times (0,T^*+1),\\
z=0&\mbox{on}\;\;\partial \Omega\times (0,T^*+1),\\
z(0)=y_0&\mbox{in}\;\;\Omega.
\end{array}\right.
\end{equation}
Finally, it follows from Proposition~\ref{Energy:2}, the fourth equation of (\ref{Optimal:5}) and
(\ref{Optimal:6}) that
\begin{equation*}
z\in C([0,T^*+1];C_0(\om))
\end{equation*}
and
\begin{eqnarray*}
\|z(T^*)\|_{L^2(\om)}&\leq&\|z(T^*)-z(T_n)\|_{L^2(\om)}+\|z(T_n)-z_n(T_n)\|_{L^2(\om)}\\
&\leq&\|z(T^*)-z(T_n)\|_{L^2(\om)}+\|z-z_n\|_{C([0,T^*+1];L^2(\om))}\rightarrow 0,
\end{eqnarray*}
which, combined with (\ref{Optimal:7}), complete the proof.
\end{proof}

At the end of this section, we give the proof of Theorem~\ref{Bang}.\\

\begin{proof}~By contradiction, there would exist an optimal control $u^*$, a positive
constant $\e_0<M$ and a measurable subset $E^*\subset (0,T^*)$ with
$|E^*|>0$ so that
\begin{equation}\ll{Bang:1}
\|u^*(t)\|_{L^q(\om)}\leq M-\e_0,\;\;\;\forall\;t\in
E^*.
\end{equation}
Let $\delta_0\in (0,|E^*|/2)$ and write
\begin{equation*}
E^*_{\delta_0}\triangleq\{t\in (0,T^*): t+\delta_0\in E^*\}.
\end{equation*}
One can easily check that
\begin{equation*}
|E^*_{\delta_0}|=|E^*\cap (\delta_0,T^*)|\geq |E^*|-\delta_0>
2^{-1}|E^*|>0.
\end{equation*}
Denote $y^*(t)\triangleq y(t;y_0,u^*)$ and
$z^*_{\delta_0}(t)\triangleq y^*(t+\delta_0)$. Then it holds that
\begin{equation}\ll{Bang:2}
\left\{
\begin{array}{lll}
(z^*_{\delta_0})_t-\Delta z^*_{\delta_0}+f(z^*_{\delta_0})=\chi_\omega u^*(t+\delta_0)
&\mbox{in}&\om\times (0,T^*-\delta_0),\\
z^*_{\delta_0}=0&\mbox{on}&\p\om\times (0,T^*-\delta_0),\\
z^*_{\delta_0}(0)=y^*(\delta_0)&\mbox{in}&\om,\\
z_{\delta_0}^*(T^*-\delta_0)=0&\mbox{in}&\om.
\end{array}\right.
\end{equation}

By Proposition~\ref{Control}, there are two positive constants
$$
\rho_0\triangleq \rho_0(E_{\delta_0}^*,T^*,\delta_0,f,\|z_{\delta_0}^*\|_{L^\infty(\om\times (0,T^*-\delta_0))})
$$
and
$$
\kappa\triangleq \kappa(E_{\delta_0}^*,T^*,\delta_0,f,\|z_{\delta_0}^*\|_{L^\infty(\om\times (0,T^*-\delta_0))}),
$$
so that
for any $w_0\in C_0(\om)$ with $\|w_0\|_{C(\o \om)}\leq \rho_0$,
there exists
a function $v\in L^\infty(0,T^*-\delta_0;L^q(\om))$ with
\begin{equation}\ll{Bang:3}
\|v\|_{L^\infty(0,T^*-\delta_0;L^q(\om))}\leq \kappa \|w_0\|_{L^2(\om)},
\end{equation}
so that $w\in C([0,T^*-\delta_0];C_0(\om))$ satisfies
\begin{equation}\ll{Bang:4}
\left\{
\begin{array}{lll}
w_t-\Delta w+f(z_{\delta_0}^*+w)-f(z_{\delta_0}^*)=\chi_\omega \chi_{E_{\delta_0}^*} v&\mbox{in}&\om\times (0,T^*-\delta_0),\\
w=0&\mbox{on}&\partial\om\times (0,T^*-\delta_0),\\
w(0)=w_0&\mbox{in}&\om,\\
w(T^*-\delta_0)=0&\mbox{in}&\om.
\end{array}\right.
\end{equation}

Now, we choose $\delta\in (0,\delta_0)$ so that
\begin{equation*}
\|y^*(\delta)-y^*(\delta_0)\|_{C(\o \om)}+\|y^*(\delta)-y^*(\delta_0)\|_{L^2(\om)}
\leq \min\{\rho_0,\kappa^{-1}\varepsilon_0\}.
\end{equation*}
This, together with (\ref{Bang:3}) and (\ref{Bang:4}), indicates that there exists a control
$v_\delta\in L^\infty(0,T^*-\delta_0;L^q(\om))$ with
\begin{equation}\ll{Bang:6}
\|v_\delta\|_{L^\infty(0,T^*-\delta_0;L^q(\om))}\leq \kappa \|y^*(\delta)-y^*(\delta_0)\|_{L^2(\om)}\leq \e_0,
\end{equation}
so that $w_\delta\in C([0,T^*-\delta_0];C_0(\om))$ satisfies
\begin{equation}\ll{Bang:7}
\left\{
\begin{array}{lll}
(w_\delta)_t-\Delta w_\delta+f(z_{\delta_0}^*+w_\delta)-f(z_{\delta_0}^*)
=\chi_\omega \chi_{E_{\delta_0}^*} v_\delta&\mbox{in}&\om\times (0,T^*-\delta_0),\\
w_\delta=0&\mbox{on}&\partial\om\times (0,T^*-\delta_0),\\
w_\delta(0)=y^*(\delta)-y^*(\delta_0)&\mbox{in}&\om,\\
w_\delta(T^*-\delta_0)=0&\mbox{in}&\om.
\end{array}\right.
\end{equation}
It follows from (\ref{Bang:2}) and (\ref{Bang:7}) that
\begin{equation}\ll{Bang:8}
\left\{
\begin{array}{lll}
(z_{\delta_0}^*+w_\delta)_t-\Delta (z_{\delta_0}^*+w_\delta)
+f(z_{\delta_0}^*+w_\delta)&&\\
\;\;\;\;\;\;\;\;\;\;\;\;\;\;\;\;\;\;\;\;\;\;\;\;\;\;\;\;\;\;\;\;\;\;
=\chi_\omega (u^*(t+\delta_0)+\chi_{E_{\delta_0}^*} v_\delta)&\mbox{in}&\om\times (0,T^*-\delta_0),\\
z_{\delta_0}^*+w_\delta=0&\mbox{on}&\partial\om\times (0,T^*-\delta_0),\\
(z_{\delta_0}^*+w_\delta)(0)=y^*(\delta)&\mbox{in}&\om,\\
(z_{\delta_0}^*+w_\delta)(T^*-\delta_0)=0&\mbox{in}&\om.
\end{array}\right.
\end{equation}
Moreover, by (\ref{Bang:1}) and (\ref{Bang:6}), the function
\begin{equation}\ll{Bang:9}
u_\delta^*(t)\triangleq u^*(t+\delta_0)+\chi_{E_{\delta_0}^*}v_\delta(t),\;\;\;t\in (0,T^*-\delta_0),
\end{equation}
satisfies
\begin{equation}\ll{Bang:10}
\|u_\delta^*(t)\|_{L^q(\om)}\leq M\;\;\mbox{for a.e.}\;\;t\in (0,T^*-\delta_0).
\end{equation}
Finally, we define
\begin{equation*}
\widetilde{u}_\delta^*(t)\triangleq\left\{
\begin{array}{ll}
u^*(t),&t\in (0,\delta],\\
u_\delta^*(t-\delta),&t\in (\delta,T^*-\delta_0+\delta],\\
0,&t\in (T^*-\delta_0+\delta,+\infty).
\end{array}\right.
\end{equation*}
Then by (\ref{Bang:8})-(\ref{Bang:10}) and Proposition~\ref{Energy:2}, we obtain that
\begin{equation*}
\widetilde{u}_\delta^*\in \mathcal{U},\;\;y(\cdot;y_0,\widetilde{u}_\delta^*)\in C([0,T^*-\delta_0+\delta];C_0(\om))\;\;
\mbox{and}\;\;y(T^*-\delta_0+\delta;y_0,\widetilde{u}_\delta^*)=0.
\end{equation*}
These contradict the optimality of $T^*$ and complete the proof.
\end{proof}

\end{document}